\documentclass[12pt]{amsart}
\usepackage{amssymb}
\usepackage{amsmath}
\usepackage{amsfonts}

\newtheorem{theorem}{Theorem}
\theoremstyle{plain}

\newtheorem{corollary}{Corollary}

\newtheorem{definition}{Definition}

\newtheorem{lemma}{Lemma}

\newtheorem{proposition}{Proposition}
\newtheorem{remark}{Remark}

\numberwithin{equation}{section}
\input{tcilatex}

\begin{document}
\title[Some integral formulae on weighted manifolds]{Some integral formulae
on weighted manifolds}
\author{Mohammed Abdelmalek}
\address[A. One ]{ Ecole Sup\'{e}rieure de Management. Tlemcen Ageria.}
\email{abdelmalekmhd@yahoo.fr}
\author{Mohammed Benalili}
\curraddr[A. Two]{Dept. Mathematics, Faculty of Sciences BP 119 Universit%
\'{e} UABB Tlemcen. Algeria.}
\email{m\_benalili@yahoo.fr}

\begin{abstract}
Introducing a notion of the weighted$\ $mean $\sigma _{r}^{\infty }-$%
curvature and using the weighted Newton transformations we derive in this
paper some integral formulae on weighted manifolds. These formulae
generalize the flux formula and some of its examples of applications
obtained by Alias, de Lira and Malacarne \cite{alm}.
\end{abstract}

\maketitle

\section{Introduction}

Several works have been done in the past to study the geometric properties
of constant $k$-mean curvature hypersurfaces in space forms, where the $k$%
-mean curvature is defined as the $k$-th elementary symmetric function of
the eigenvalues of the second fundamental form see (\cite{alm}). Motivated
in part by connections with the Ricci flow, much works have also been done
on geometric properties of manifolds and hypersurfaces when the manifold is
endowed with a "weighted" volume element; i.e. one integrates using the
smooth measure $e^{-f}dvol_{g}$ for $dvol_{g}$ the Riemannian volume element
of the metric $g$ see ( \cite{C}, \cite{Co}, \cite{cr}, \cite{e}, \cite{m}
). In this work, we follow Case (\cite{C}) to introduce the notion of
weighted $k$-mean curvature and using the weighted Newtons transformations
introduced in (\cite{C}) we obtain an integral formula on weighted manifolds
and give some applications. This latter formula was first introduced by
Kusner in (\cite{K}) and nowadays it's called the flux formula. Then it was
extended to $k$-curvature in a nice paper by Al\'{\i}as, de Lira, and
Malacarne see (\cite{alm}). Where they studied the properties of certain
geometrical configurations, more particularly they established a flux
formula and gave examples of geometric applications. Our paper extends some
properties obtained by the authors cited above for weighted manifolds.

\section{Preliminaries}

In this section we fix the notations and recall some definitions and
properties of the weighted symmetric functions and the weighted Newton
transformations: for more details see (\cite{C},\cite{m}).

Given a complete $n$-dimensional Riemannian manifold $(M,\left\langle
,\right\rangle )$ and a smooth function $f:M\longrightarrow 
%TCIMACRO{\U{211d} }%
%BeginExpansion
\mathbb{R}
%EndExpansion
.$ The weighted manifold $M_{f}$ associate to $M$ is the triplet $%
(M,\left\langle ,\right\rangle ,dv_{f}),$where $dv_{f}=e^{-f}dv$ and $dv$ is
the standard volume element of $M.$

\bigskip Consider the tensional connection 
\begin{equation*}
\widetilde{\nabla }_{X}Y=\nabla _{X}Y+\left\langle X,Y\right\rangle
V-\left\langle V,Y\right\rangle X
\end{equation*}%
where $V=\left\langle \nabla f,\nu \right\rangle \nu $, $\nu $ is a vector
field on $M$ \ orthogonal to $X$ and $\nabla $ stands for the covariant
derivative on $M.$ This connection is one of the three basic types of metric
connections introduced by Elie Cartan. It was studied by I. Agricola and M.
Kraus \cite{ak}.

If $B$ and $\widetilde{B}$ are second forms on $M$ defined by $B\left(
X,Y\right) =\left\langle \nabla _{X}Y,\nu \right\rangle $ and $\widetilde{B}%
\left( X,Y\right) =\left\langle \widetilde{\nabla }_{X}Y,\nu \right\rangle $%
, we have 
\begin{equation*}
\widetilde{B}\left( X,Y\right) =B\left( X,Y\right) +\left\langle
X,Y\right\rangle \left\langle \nabla f,\nu \right\rangle
\end{equation*}%
An example of this situation is as follows: let$\ \psi _{t}:M^{n}\rightarrow
(\overline{M}^{n+1}$,$\left\langle ,\right\rangle )$ be a one family of
parameter of immersions of an $n-$dimensional manifold $M^{n}$ into an $n+1$%
- $\overline{M}^{n+1}$ Riemannian manifold where $\nabla $ is the connection
induced on $M^{n}$ and $\widetilde{\nabla }$ is the the tonsorial connection
on $M^{n}$, so the mean curvatures of $M$ with respective to the connections 
$\nabla $ and $\widetilde{\nabla }$ are related by%
\begin{equation*}
n\widetilde{H}=nH+\left\langle \nabla f,\nu \right\rangle
\end{equation*}%
which is the classical weighted mean curvature.

The important point is that hypersurfaces of constant weighted curvature
appear as critical points of certain weighted volume functionals. The
fundamental analogy with constant mean curvature hypersurfaces draws the
attention in this area. In terms of matrices we have 
\begin{equation*}
\widetilde{B}=B+\left\langle \nabla f,\nu \right\rangle I
\end{equation*}%
so if $\widetilde{\tau }_{i}$ and $\tau _{i}$ are the eigenvalues of $%
\widetilde{B}$ and $B$ respectively, we get%
\begin{equation*}
\widetilde{\tau }_{i}=\tau _{i}+\left\langle \nabla f,\nu \right\rangle .
\end{equation*}%
Now, putting $\lambda =\left\langle \nabla f,\nu \right\rangle ,$ we obtain 
\cite{abn} 
\begin{equation}
\widetilde{\sigma }_{k}=\sigma _{k}\left( \widetilde{B}\right) =\underset{j=0%
}{\overset{k}{\sum }}\binom{n-k+j}{j}\lambda ^{j}\sigma _{k-j}\left(
B\right) .  \label{1}
\end{equation}%
where $\widetilde{\sigma }_{k}$ stand for the symmetric functions of $%
\widetilde{B}.$

Let $\overline{M}$ be an $(n+1)$-dimensional Riemannian manifold, and $\psi
:M\longrightarrow \overline{M}$ \ be an isometrically immersed hypersurface
with $\triangledown $ and $\overline{\triangledown }$ the Levi-Civita
connections on $M$ and $\overline{M}$ respectively. The Weingarten formula
of this immersion is written as follows%
\begin{equation}
\text{ }AX=-\left( \overline{\triangledown }_{X}N\right) ^{\intercal } 
\notag
\end{equation}%
where $A$ is the shape operator of the hypersurface $M$ with respect to the
Gauss map $N,$ and $^{\top }$ denotes the orthogonal projection on the
tangent vector bundle of $M.$ As it is well known $A$ is a linear self
adjoint operator and at each point $p\in M,$ its eigenvalues $\mu
_{1},...,\mu _{n}$ are the principal curvatures of $M.$

Associate to the shape operator $A$ are the weighted elementary symmetric
functions $\sigma _{k}^{\infty }:%
%TCIMACRO{\U{211d} }%
%BeginExpansion
\mathbb{R}
%EndExpansion
\times 
%TCIMACRO{\U{211d} }%
%BeginExpansion
\mathbb{R}
%EndExpansion
^{n}\longrightarrow 
%TCIMACRO{\U{211d} }%
%BeginExpansion
\mathbb{R}
%EndExpansion
$ defined recursively ( and introduced by Case in \cite{C} ) by 
\begin{equation}
\left\{ 
\begin{array}{l}
\sigma _{0}^{\infty }(\mu _{0},\mu )=1, \\ 
k\sigma _{k}^{\infty }(\mu _{0},\mu )=\sigma _{k-1}^{\infty }(\mu _{0},\mu )%
\underset{j=0}{\overset{n}{\dsum }}\mu _{j}+\underset{i=1}{\overset{k-1}{%
\dsum }}\underset{j=1}{\overset{n}{\dsum }}\left( -1\right) ^{i}\sigma
_{k-1-i}^{\infty }(\mu _{0},\mu )\mu _{j}^{i}\text{ \ \ \ \ \ for }k\geq 1%
\end{array}%
\right.  \notag
\end{equation}%
where $\mu _{0}\in 
%TCIMACRO{\U{211d} }%
%BeginExpansion
\mathbb{R}
%EndExpansion
$ and $\mu =\left( \mu _{1},...,\mu _{n}\right) \in 
%TCIMACRO{\U{211d} }%
%BeginExpansion
\mathbb{R}
%EndExpansion
^{n}$. In particular for $\mu _{0}=0$ we recover $\sigma _{k}^{\infty
}(0,\mu )=\sigma _{k}(\mu )$ the classical elementary symmetric functions
defined in \cite{r}.

\begin{definition}
(\cite{C})The weighted Newton transformations (W.N.T) $T_{k}^{\infty }(\mu
_{0},A)$ are defined inductively from $A$ by :%
\begin{equation}
\left\{ 
\begin{array}{l}
T_{k}^{\infty }(\mu _{0},A)=I \\ 
T_{k}^{\infty }(\mu _{0},A)=\sigma _{k}^{\infty }(\mu
_{0},A)I-AT_{k-1}^{\infty }(\mu _{0},A)\text{\ \ \ \ for }k\geq 1%
\end{array}%
\right.  \notag
\end{equation}%
or equivalently 
\begin{equation}
\text{ }T_{k}^{\infty }(\mu _{0},A)=\underset{j=0}{\overset{k}{\dsum }}%
\left( -1\right) ^{j}\sigma _{k-j}^{\infty }(\mu _{0},A)A^{j}  \notag
\end{equation}%
where $I$ stands for the identity on the Lie algebra of vector fields $%
\varkappa (M)$, $\sigma _{k}^{\infty }(\mu _{0},A)=\sigma _{k}^{\infty }(\mu
_{0},\mu _{1},...,\mu _{n})$ and $\mu _{1},...,\mu _{n}$ are the eigenvalues
of $A$.
\end{definition}

It should be noted that $T_{k}^{\infty }(0,A)=T_{k}(A)$ is the classical
Newton transformations introduced in \cite{r}.

These functions enjoy the nice following properties.

\begin{proposition}
\cite{C} For $\mu _{0},\mu _{1}\in 
%TCIMACRO{\U{211d} }%
%BeginExpansion
\mathbb{R}
%EndExpansion
$ and $\mu \in 
%TCIMACRO{\U{211d} }%
%BeginExpansion
\mathbb{R}
%EndExpansion
^{n}$, we have%
\begin{equation}
\sigma _{k}^{\infty }(\mu _{0}+\mu _{1},\mu )=\sum_{j=0}^{k}\frac{\mu
_{1}^{j}}{j!}\sigma _{k-j}^{\infty }(\mu _{0},\mu ).  \notag
\end{equation}%
In particular,%
\begin{equation}
\sigma _{k}^{\infty }(\mu _{1},\mu )=\sum_{j=0}^{k}\frac{\mu _{1}^{j}}{j!}%
\sigma _{k-j}\left( \mu \right)  \label{2}
\end{equation}%
\begin{equation}
\text{trace}(AT_{k}^{\infty }\left( (\mu _{0},\mu )\right) )=(k+1)\sigma
_{k+1}^{\infty }(\mu _{0},\mu )-\mu _{0}\sigma _{k}^{\infty }(\mu _{0},\mu ).
\label{2'}
\end{equation}%
For $i\in \left\{ 1,...,n\right\} $ we have%
\begin{equation}
\sigma _{k,i}^{\infty }(\mu _{0},\mu )=\sigma _{k}^{\infty }(\mu _{0},\mu
)-\mu _{i}\sigma _{k-1,i}^{\infty }\left( \mu _{0},\mu \right)  \notag
\end{equation}%
and the $i^{th}$ eigenvalue of $T_{k}^{\infty }(\mu _{0},\mu )$ is equal to $%
\sigma _{k,i}^{\infty }(\mu _{0},\mu )$ where $\sigma _{k,i}^{\infty }\left(
\mu _{0},\mu \right) =\sigma _{k}^{\infty }\left( \mu _{0},\mu _{1},...,\mu
_{i-1},\mu _{i+1},...,\mu _{n}\right) .$
\end{proposition}

We can see by (\ref{1}) and (\ref{2}) that $\widetilde{\sigma }_{k}$ and $%
\sigma _{k}^{\infty }(\mu _{0},\mu )$ are polynomials of the same degree but
with slightly different coefficients, which differ by a multiplicative
constant only.

\begin{definition}
The weighted $k^{th}$ mean curvature $H_{k,f}$ is given by:%
\begin{equation}
\text{ }\left( 
\begin{array}{c}
n \\ 
k%
\end{array}%
\right) H_{k,f}=\sigma _{k}^{\infty }(\left\langle \nu ,\triangledown
f\right\rangle ,A)  \notag
\end{equation}%
\bigskip where $\nu $ is the unit outpointing vector field normal to $M$ in $%
\overline{M}.$
\end{definition}

\begin{remark}
In particular for $k=1$ and in view of formula \ref{2}, we get 
\begin{equation}
nH_{1,f}=\sigma _{1}^{\infty }(\left\langle \nu ,\triangledown
f\right\rangle ,A)=\sigma _{1}\left( A\right) +\left\langle \nu
,\triangledown f\right\rangle =nH+\left\langle \nu ,\triangledown
f\right\rangle  \notag
\end{equation}%
which is the classical definition of the weighted mean curvature of the
hypersurface $M$ studied by Gromov \cite{G}.
\end{remark}

To clarify this notion of curvature we will study the case $k=2$.

Consider a one family of parameter $\ \psi _{t}:M^{n}\rightarrow \overline{M}%
^{n+1}\left( c\right) $ of immersions of an $n-$dimensional closed manifold $%
M^{n}$ into an $n+1$ space form $\left( \overline{M}^{n+1}\text{,}%
\left\langle ,\right\rangle \right) $ of constant curvature $c$. Denote by $%
X $ the deformation vector field and by $\nu $ the normal vector field to $%
\overline{M}^{n+1}$. Put $\lambda =\left\langle X,\nu \right\rangle .$

Consider the variational problem%
\begin{equation}
\delta \left( \int_{M}\sigma _{1}^{\infty }dV_{f}\right) =0  \label{3.1}
\end{equation}%
that is to say%
\begin{eqnarray}
\delta \left( \int_{M}\sigma _{1}^{\infty }dV_{f}\right) &=&\frac{d}{dt}%
\left( \int_{M}\sigma _{1}^{\infty }dV_{f}\right)  \label{3.2} \\
&=&\frac{d}{dt}\left( \int_{M}\left( \sigma _{1}+\mu \right) dV_{f}\right) 
\notag \\
&=&\int_{M}\left( \frac{d\sigma _{1}}{dt}+\frac{d\mu }{dt}\right)
dV_{f}+\int_{M}\left( \sigma _{1}+\mu \right) \frac{d\left( dV_{f}\right) }{%
dt}.  \notag
\end{eqnarray}%
Now, by formula (9c) in \ page 469, we have%
\begin{equation*}
\frac{d\sigma _{1}}{dt}=\lambda \left( \sigma _{1}^{2}-2\sigma _{2}\right)
+\lambda _{,ii}+\sigma _{1,j}X^{j}+nc\lambda
\end{equation*}%
and by the well known fact $\frac{d}{dt}dV=(-\lambda \sigma
_{1}+X_{,j}^{j})dV$, we infer%
\begin{equation*}
\frac{d\left( dV_{f}\right) }{dt}=(-\lambda \sigma
_{1}-X^{j}f_{,j}+X_{,j}^{j})dV_{f}
\end{equation*}%
and 
\begin{equation*}
\frac{d\mu }{dt}=\left\langle \nabla \nabla _{X}f,\nu \right\rangle
-\left\langle \nabla f,\left[ \nu ,X\right] \right\rangle
\end{equation*}%
By the definition of the weighted divergence, we have $%
div_{f}X=-X^{j}f_{,j}+X_{,j}^{j}$ and replacing in (\ref{3.2}), we get

\begin{eqnarray*}
\frac{d}{dt}\left( \int_{M}\sigma _{1}^{\infty }dV_{f}\right)
&=&\int_{M}\left( \lambda \left( \sigma _{1}^{2}-2\sigma _{2}\right)
+\lambda _{,ii}+\sigma _{1,j}X^{j}++nc\lambda +\left\langle \nabla \nabla
_{X}f,\nu \right\rangle -\left\langle \nabla f,\left[ \nu ,X\right]
\right\rangle \right) dV_{f} \\
&&+\int_{M}\left( \sigma _{1}+\mu \right) (-\lambda \sigma
_{1}+div_{f}\left( X\right) )dV_{f} \\
&=&\int_{M}\left( -2\lambda \sigma _{2}+\lambda _{,ii}+\left\langle \nabla
\nabla _{X}f,\nu \right\rangle -\left\langle \nabla f,\left[ \nu ,X\right]
\right\rangle \right) dV_{f} \\
&&+\int_{M}\left( -\lambda \mu \sigma _{1}+\mu div_{f}\left( X\right)
+nc\lambda \right) dV_{f} \\
&=&\int_{M}\left( -\lambda \left( 2\sigma _{2}^{\infty }-\mu _{1}\sigma
_{1}^{\infty }\right) +\lambda _{,ii}+\left\langle \nabla \nabla _{X}f,\nu
\right\rangle -\left\langle \nabla f,\left[ \nu ,X\right] \right\rangle
\right) dV_{f} \\
&&+\int_{M}\left( \mu div_{f}\left( X\right) +nc\lambda \right) dV_{f}.
\end{eqnarray*}

we have

\begin{theorem}
The Euler-Lagrange equation corresponding to the problem (\ref{3.1}) is 
\begin{equation*}
-\lambda \left( 2\sigma _{2}^{\infty }-\mu _{1}\sigma _{1}^{\infty }\right)
+\lambda _{,ii}+\left\langle \nabla \nabla _{X}f,\nu \right\rangle
-\left\langle \nabla f,\left[ \nu ,X\right] \right\rangle +\mu div_{f}\left(
X\right) +nc\lambda =0.
\end{equation*}
\end{theorem}

To clarify the idea we will consider simpler cases: \ put $f(x)=\frac{1}{2}%
\left\Vert x\right\Vert ^{2.}$, and $X=g\nu $, where $g$ is a $C^{2}$%
-function, so $\lambda =g$ and $\mu =\left\langle x,\nu \right\rangle $ the
support function. We get 
\begin{eqnarray*}
div_{f}\left( X\right) &=&-\left( g\nu ^{j}\right) f_{,j}+\left( g\nu
^{j}\right) _{,j} \\
&=&-g\left\langle x,\nu \right\rangle +g_{j}\nu ^{j}+g\nu _{,j}^{j} \\
&=&-g\mu +\nu (g)-g\sigma _{1} \\
&=&g\sigma _{1}^{\infty }+\nu \left( g\right) ,
\end{eqnarray*}%
\begin{eqnarray*}
\int_{M}g_{,ii}dV_{f} &=&\int_{M}g\left( f_{,ii}-f_{i}f^{,i}\right) dV_{f} \\
&=&\int_{M}g\left( n-2f\right) dV_{f}
\end{eqnarray*}%
and%
\begin{equation*}
\left\langle \nabla \nabla _{X}f,\nu \right\rangle -\left\langle \nabla f, 
\left[ \nu ,X\right] \right\rangle =g\nu \left( \mu _{1}\right) .
\end{equation*}%
Hence%
\begin{eqnarray*}
\frac{d}{dt}\left( \int_{M}\sigma _{1}^{\infty }dV_{f}\right)
&=&\int_{M}\left( -g\left( 2\sigma _{2}^{\infty }-2\mu \sigma _{1}^{\infty
}\right) +g\left( n-2f\right) +g\nu \left( \mu \right) \right) dV_{f} \\
&&+\int_{M}\mu \nu \left( g\right) dV_{f}
\end{eqnarray*}%
and since 
\begin{eqnarray*}
\int_{M}\mu \nu \left( g\right) dV_{f} &=&-\int_{M}\nu (\mu e^{-f})gdV \\
&=&-\int_{M}g\left( \nu \left( \mu \right) -\mu ^{2}\right) dV_{f}
\end{eqnarray*}%
we deduce%
\begin{equation*}
\frac{d}{dt}\left( \int_{M}\sigma _{1}^{\infty }dV_{f}\right) =\int_{M}g%
\left[ -2\sigma _{2}^{\infty }+2\mu \sigma _{1}^{\infty }-2f+\mu
^{2}+n\left( 1+c\right) \right] dV_{f}\text{.}
\end{equation*}

\begin{corollary}
Under the above assumptions the Euler-Langrange equation of the problem (\ref%
{3.1}) is given by 
\begin{equation*}
-2\sigma _{2}^{\infty }+2\mu \sigma _{1}^{\infty }-2f+\mu ^{2}+n\left(
1+c\right) =0.
\end{equation*}
\end{corollary}

Example: Consider a hypersurface $M_{n}$ of the unit sphere $\ S^{n+1}$. \
The Euler-Lagrange is then written%
\begin{equation*}
-\sigma _{2}^{\infty }+\sigma _{1}^{\infty }+n=0
\end{equation*}%
or equivalently%
\begin{equation*}
\left( 2-\underset{j=0}{\dsum^{n}\mu _{j}}\right) \sigma _{1}^{\infty }+2n+%
\underset{j=0}{\overset{n}{\dsum }}\mu _{j}=0
\end{equation*}%
where $\mu _{j}$ stand for the eigenvalues of the second fundamental form.

Or%
\begin{equation*}
\left( \sigma _{1}^{\infty }\right) ^{2}-4\sigma _{1}^{\infty }-2n-1=0
\end{equation*}%
which gives 
\begin{equation*}
\sigma _{1}^{\infty }=2\pm \sqrt{2n+3}
\end{equation*}%
or%
\begin{equation*}
\sigma _{1}=1\pm \sqrt{2n+3}.
\end{equation*}%
We can cite a candidate to our situation: Cliffor torus in $S^{n+1}$i.e. $%
M^{n}$ is a product of spheres $S^{n_{1}}(r_{1})\times S^{n_{2}}(r_{2})$ , $%
n_{1}+n_{2}=n$ of appropriate radii $r_{1}$, $r_{2}.$ An $H(r)$-torus in $%
S^{n+1}$ is obtained by the canonical immersions $S^{n-1}\left( r\right)
\subset R^{n},S^{1}(\sqrt{1-r^{2}})\subset R^{2}$, $0<r<1$, \ as $%
S^{n-1}\left( r\right) \times S^{1}(\sqrt{1-r^{2}})\subset S^{n+1}$. The
principal curvatures are given, for a chosen orientation, by 
\begin{equation*}
\mu _{1}=...=\mu _{n-1}=\frac{\sqrt{1-r^{2}}}{r}\text{, }\mu _{n}=-\frac{r}{%
\sqrt{1-r^{2}}}
\end{equation*}%
so for the $H(r)$-torus 
\begin{equation*}
\sigma _{1}=\left( n-1\right) \frac{\sqrt{1-r^{2}}}{r}-\frac{r}{\sqrt{1-r^{2}%
}}=\frac{n(1-r^{2})-1}{r\sqrt{1-r^{2}}}.
\end{equation*}%
For the torus to be an example we have to show that 
\begin{equation}
\frac{n(1-r^{2})-1}{r\sqrt{1-r^{2}}}=1+\sqrt{2n+3}  \tag{3.3}
\end{equation}%
has a root. For this, we consider the continuous function $\phi \left(
r\right) =-(1+\sqrt{2n+3})$ with $0<r<\frac{n-1}{n}$ and $n\geq 2.$ We
notice that $\lim_{r\rightarrow 0^{+}}\phi (r)=+\infty $ and $\phi \left( 
\frac{n-1}{n}\right) =\frac{n}{\sqrt{2n+1}}-1-\sqrt{2n+3}<0.$ Consequently
the equation (\ref{3.3}) admits at least one root.

\begin{definition}
We say that an hypersurface $M$ of $\overline{M}$ is $\sigma _{r}^{\infty }$%
-minimal, if $H_{r,f}=0.$ In particular $M$ is $f$-minimal if $H_{f}=-\frac{1%
}{n}\left\langle \nu ,\triangledown f\right\rangle .$
\end{definition}

Here we need $\mu _{0}=\left\langle \nabla f,\nu \right\rangle $. For the
safe of brevity, we put $T_{k}^{\infty }=T_{k}^{\infty }(\mu _{0},A)$ and $%
\sigma _{k}^{\infty }=\sigma _{k}^{\infty }(\mu _{0},A).$

The weighted divergence of the weighted Newton transformations is define by 
\begin{equation}
\func{div}_{f}T_{k}^{\infty }=e^{f}\func{div}\left( e^{-f}T_{k}^{\infty
}\right)  \notag
\end{equation}%
where%
\begin{equation}
\func{div}\left( T_{k}^{\infty }\right) =\text{trace}\left( \triangledown
T_{k}^{\infty }\right) =\underset{j=0}{\overset{k}{\dsum }}\triangledown
_{e_{i}}\left( T_{k}^{\infty }\right) \left( e_{i}\right)  \notag
\end{equation}%
and $\left\{ e_{1},...,e_{n}\right\} $\ is an orthonormal basis of the
tangent space of $M.$

\begin{lemma}
\label{lem1'}%
\begin{equation*}
tr\left( T_{k-1}^{\infty }\circ \nabla _{v}A\right) =\left\langle \nabla
\sigma _{k}^{\infty }(\mu _{1},A)-\sigma _{k-1}^{\infty }\left( \mu
_{1},A\right) \nabla \mu _{1},v\right\rangle .
\end{equation*}
\end{lemma}

\begin{proof}
The computations will be in a basis that diagonalizes $A$. Let $A=\left( 
\begin{array}{ccc}
\lambda _{1} &  &  \\ 
& . &  \\ 
&  & \lambda _{n}%
\end{array}%
\right) .$ Since the eigenvalues of $T_{k-1}^{\infty }$ are given by 
\begin{eqnarray*}
t_{i} &=&\sigma _{k-1}^{\infty }\left( \mu _{1},\lambda _{1},...,\lambda
_{i-1},\lambda _{i+1},...,\lambda _{n}\right) \\
&=&\dsum\limits_{j=0}^{k-1}\frac{\mu _{1}^{j}}{j!}\sigma _{k-1-j}\left(
\lambda _{1},...,\lambda _{i-1},\lambda _{i+1},...,\lambda _{n}\right) \\
&=&\dsum\limits_{j=0}^{k-1}\frac{\mu _{1}^{j}}{j!}\frac{\partial }{\partial
\lambda _{i}}\sigma _{k-1-j}\left( \lambda _{1},...,\lambda _{n}\right) \\
&=&\dsum\limits_{j=0}^{k-1}\frac{\mu _{1}^{j}}{j!}\dsum\limits_{i\neq
i_{j},\ i_{1}<...<i_{k-1-j}}\lambda _{i_{1}}...\lambda _{i_{k-1-j}}.
\end{eqnarray*}%
So%
\begin{equation*}
tr\left( T_{k-1}^{\infty }\circ \nabla _{v}A\right)
=\dsum\limits_{i=1}^{n}t_{i}\nabla _{v}\lambda _{i}
\end{equation*}%
\begin{eqnarray*}
&=&\dsum\limits_{j=0}^{k}\frac{\mu _{1}^{j}}{j!}\nabla
_{v}\dsum\limits_{i_{1}<...<i_{k-j}}\lambda _{i_{1}}...\lambda _{i_{k-j}} \\
&=&\nabla _{v}\sigma _{k}^{\infty }(\mu _{1},A)-\nabla _{v}\mu
_{1}\dsum\limits_{j=0}^{k-1}\frac{\mu _{1}^{j}}{j!}\sigma _{k-1-j}\left(
A\right) \\
&=&\nabla _{v}\sigma _{k}^{\infty }(\mu _{1},A)-\left( \nabla _{v}\mu
_{1}\right) \sigma _{k-1}^{\infty }\left( \mu _{1},A\right) .
\end{eqnarray*}
\end{proof}

\begin{lemma}
\label{lem1} The weighted divergence of the weighted Newton transformations $%
T_{k}^{\infty }$ are inductively given by the following formula 
\begin{equation*}
\func{div}_{f}T_{0}^{\infty }=\triangledown f
\end{equation*}%
and%
\begin{equation}
\func{div}_{f}T_{k}^{\infty }=\sigma _{k}^{\infty }\triangledown f+\sigma
_{k-1}^{\infty }\left( \mu _{1},A\right) \nabla \mu _{1}-A\func{div}%
_{f}T_{k-1}^{\infty }-\sum_{i=1}^{n}\left( \overline{R}(N,T_{k-1}^{\infty
}(e_{i}))e_{i}\right) ^{\top }\text{\ \ \ \ \ for }k\geq 1\text{.}  \notag
\end{equation}
\end{lemma}

\begin{proof}
We have%
\begin{eqnarray*}
\func{div}_{f}T_{k}^{\infty } &=&e^{f}\func{div}\left( e^{-f}T_{k}^{\infty
}\right) \\
&=&e^{f}\overset{n}{\underset{i=1}{\sum }}\left[ \triangledown
_{e_{i}}(e^{-f}T_{k}^{\infty })(e_{i})\right] \\
&=&e^{f}\overset{n}{\underset{i=1}{\sum }}\left[ \triangledown
_{e_{i}}(e^{-f}T_{k}^{\infty }(e_{i}))-e^{-f}T_{k}^{\infty }(\triangledown
_{e_{i}}e_{i})\right] \\
&=&e^{f}\overset{n}{\underset{i=1}{\sum }}\left[ e^{-f}\triangledown
_{e_{i}}(T_{k}^{\infty }(e_{i}))+e_{i}\left( e^{-f}\right) T_{k}^{\infty
}(e_{i})-e^{-f}T_{k}^{\infty }(\triangledown _{e_{i}}e_{i})\right] \\
&=&e^{f}\overset{n}{\underset{i=1}{\sum }}\left[ e^{-f}\triangledown
_{e_{i}}(T_{k}^{\infty }(e_{i}))-e^{-f}\left\langle \triangledown
f,e_{i}\right\rangle T_{k}^{\infty }(e_{i})-e^{-f}T_{k}^{\infty
}(\triangledown _{e_{i}}e_{i})\right] \\
&=&\func{div}T_{k}^{\infty }-T_{k}^{\infty }\left( \triangledown f\right)
\end{eqnarray*}%
It is not difficult to see that 
\begin{equation*}
\text{div}_{M}T_{k}^{\infty }=\nabla \sigma _{k}^{\infty }-A_{k}\text{div}%
_{M}T_{k-1}^{\infty }-\sum_{i=1}^{n}(\nabla _{e_{i}}A_{k})(T_{k-1}^{\infty
}(e_{i}))\text{.}
\end{equation*}%
Indeed, 
\begin{eqnarray*}
(\nabla _{e_{i}}T_{k}^{\infty })(e_{i}) &=&\nabla _{e_{i}}(T_{k}^{\infty
}(e_{i}))-T_{k}^{\infty }(\nabla _{e_{i}}e_{i}) \\
&=&\nabla _{e_{i}}((\sigma _{k}^{\infty }I-AT_{k-1}^{\infty })e_{i})-(\sigma
_{k}^{\infty }I-AT_{k-1}^{\infty })(\nabla _{e_{i}}e_{i}) \\
&=&\nabla _{e_{i}}(\sigma _{k}^{\infty }e_{i})-\nabla
_{e_{i}}(AT_{k-1}^{\infty }(e_{i}))-\sigma _{k}^{\infty }(\nabla
_{e_{i}}e_{i})+(AT_{k-1}^{\infty })(\nabla _{e_{i}}e_{i}) \\
&=&e_{i}(\sigma _{k}^{\infty })e_{i}+\sigma _{k}^{\infty }(\nabla
_{e_{i}}e_{i})-\nabla _{e_{i}}(AT_{k-1}^{\infty }(e_{i}))-\sigma
_{k}^{\infty }(\nabla _{e_{i}}e_{i})+(AT_{k-1}^{\infty })(\nabla
_{e_{i}}e_{i}) \\
&=&e_{i}(\sigma _{k}^{\infty })e_{i}-(\nabla _{e_{i}}(AT_{k-1}^{\infty
}(e_{i}))-(AT_{k-1}^{\infty })(\nabla _{e_{i}}e_{i})) \\
&=&(d_{e_{i}}\sigma _{k}^{\infty })(e_{i})-(\nabla _{e_{i}}(AT_{k-1}^{\infty
}))e_{i} \\
&=&\langle \nabla \sigma _{k}^{\infty },e_{i}\rangle e_{i}-(\nabla
_{e_{i}}A)(T_{k-1}^{\infty }(e_{i}))-A((\nabla _{e_{i}}T_{k-1}^{\infty
})(e_{i}))
\end{eqnarray*}%
Thus, 
\begin{eqnarray*}
\text{div}_{M}T_{k}^{\infty } &=&\sum_{i=1}^{n}(\nabla _{e_{i}}T_{k}^{\infty
})(e_{i}) \\
&=&\sum_{i=1}^{n}\langle \nabla \sigma _{k}^{\infty },e_{i}\rangle
e_{i}-\sum_{i=1}^{n}(\nabla _{e_{i}}A)(T_{k-1}^{\infty
}(e_{i}))-\sum_{i=1}^{n}A((\nabla _{e_{i}}T_{k-1}^{\infty })(e_{i})) \\
&=&\nabla \sigma _{k}^{\infty }-Adiv_{M}T_{k-1}^{\infty
}-\sum_{i=1}^{n}(\nabla _{e_{i}}A)(T_{k-1}^{\infty }(e_{i}))
\end{eqnarray*}%
The Godazzi equation and the fact that $\nabla _{e_{i}}A$ is a self-ajoint
operator allow us to write, \ 
\begin{equation*}
\dsum\limits_{i=1}^{n}\langle (\nabla _{e_{i}}A)T_{k-1}^{\infty
}(e_{i}),v\rangle =\langle \sum_{i=1}^{n}\left( \overline{R}%
(N,T_{k-1}^{\infty }(e_{i}))e_{i}\right) ^{\top },v\rangle
+trace(T_{k-1}^{\infty }\circ \nabla _{v}A)
\end{equation*}%
where $v$ is an arbitrary vector tangent $M$.

Thus,%
\begin{equation*}
\left\langle div_{M}T_{k}^{\infty },v\right\rangle =\langle \nabla \sigma
_{k}^{\infty },v\rangle -\langle Adiv_{M}T_{k-1}^{\infty },v\rangle -\langle
\sum_{i=1}^{n}\left( \overline{R}(N,T_{k-1}^{\infty }(e_{i}))e_{i}\right)
^{\top },v\rangle -tr(T_{k-1}^{\infty }\circ \nabla _{v}A)
\end{equation*}%
Using now Lemma \ref{lem1'}%
\begin{equation*}
tr(T_{k-1}^{\infty }\circ \nabla _{v}A)=\left\langle \nabla \sigma
_{k}^{\infty }(\mu _{1},A)-\left( \nabla \mu _{1}\right) \sigma
_{k-1}^{\infty }\left( \mu _{1},A\right) ,v\right\rangle
\end{equation*}%
we have,%
\begin{equation*}
\left\langle div_{M}T_{k}^{\infty },v\right\rangle =\sigma _{k-1}^{\infty
}\left( \mu _{1},A\right) \left\langle \nabla \mu _{1},v\right\rangle
-\langle Adiv_{M}T_{k-1}^{\infty },v\rangle -\langle \sum_{i=1}^{n}\left( 
\overline{R}(N,T_{k-1}^{\infty }(e_{i}))e_{i}\right) ^{\top },v\rangle
\end{equation*}%
or equivalently,%
\begin{equation*}
div_{M}T_{k}^{\infty }=\sigma _{k-1}^{\infty }\left( \mu _{1},A\right)
\nabla \mu _{1}-Adiv_{M}T_{k-1}^{\infty }-\sum_{i=1}^{n}\left( \overline{R}%
(N,T_{k-1}^{\infty }(e_{i}))e_{i}\right) ^{\top }.
\end{equation*}%
Finally%
\begin{eqnarray*}
\func{div}_{f}T_{k}^{\infty } &=&T_{k}^{\infty }\left( \triangledown
f\right) +\sigma _{k-1}^{\infty }\left( \mu _{1},A\right) \nabla \mu
_{1}-Adiv_{M}T_{k-1}^{\infty }-\sum_{i=1}^{n}\left( \overline{R}%
(N,T_{k-1}^{\infty }(e_{i}))e_{i}\right) ^{\top } \\
&=&\left( \sigma _{k}^{\infty }I-AT_{k-1}^{\infty }\right) \left(
\triangledown f\right) +\sigma _{k-1}^{\infty }\left( \mu _{1},A\right)
\nabla \mu _{1}-A\left( \func{div}_{f}T_{k-1}^{\infty }-T_{k-1}^{\infty
}\left( \triangledown f\right) \right) -\sum_{i=1}^{n}\left( \overline{R}%
(N,T_{k-1}^{\infty }(e_{i}))e_{i}\right) ^{\top } \\
&=&\sigma _{k}^{\infty }\triangledown f+\sigma _{k-1}^{\infty }\left( \mu
_{1},A\right) \nabla \mu _{1}-\left( AT_{k-1}^{\infty }\right) \left(
\triangledown f\right) -A\left( \func{div}_{f}T_{k-1}^{\infty }\right)
+\left( AT_{k-1}^{\infty }\right) \left( \triangledown f\right)
-\sum_{i=1}^{n}\left( \overline{R}(N,T_{k-1}^{\infty }(e_{i}))e_{i}\right)
^{\top } \\
&=&\sigma _{k}^{\infty }\triangledown f+\sigma _{k-1}^{\infty }\left( \mu
_{1},A\right) \nabla \mu _{1}-A\func{div}_{f}T_{k-1}^{\infty
}-\sum_{i=1}^{n}\left( \overline{R}(N,T_{k-1}^{\infty }(e_{i}))e_{i}\right)
^{\top }\text{.}
\end{eqnarray*}%
Which achieves the proof of Lemma \ref{lem1}.
\end{proof}

\begin{corollary}
\label{cor1} If $\overline{M}$ has constant sectional curvature, then 
\begin{equation*}
\func{div}_{f}T_{k}^{\infty }=T_{k}^{\infty }\left( \triangledown f\right)
+\sigma _{k-1}^{\infty }\left( \mu _{1},A\right) \nabla \mu _{1}.
\end{equation*}
\end{corollary}

\begin{proof}
\textit{\ If }$\overline{M}$ has constant sectional curvature, then $\left( 
\overline{R}(N,T_{k-1}^{\infty }(e_{i}))e_{i}\right) ^{\top }=0,$ and we have%
\textit{\ }%
\begin{equation*}
\func{div}_{f}T_{k}^{\infty }=\sigma _{k}^{\infty }\triangledown f-A\func{div%
}_{f}T_{k-1}^{\infty }+\sigma _{k-1}^{\infty }\left( \mu _{1},A\right)
\nabla \mu _{1}.
\end{equation*}%
The desired relation results by a recursive argument.
\end{proof}

\section{Main results}

The aim of this part is to derive an integral formula on weighted manifolds
with constant sectional curvature and to give some of its geometric
applications. The method is based on the computation of the weighted
divergence $\func{div}_{f}\left( T_{k}^{\infty }Y^{\intercal }\right) $ and $%
\left\langle \func{div}_{f}T_{k}^{\infty },Y\right\rangle $, where $Y$ is a
conformal vector field. To do so, we first consider the following geometric
configuration: let $\overline{M}^{n+1}$ be an oriented Riemannian manifold
with metric $\left\langle ,\right\rangle $, $P^{n}\subset \overline{M}^{n+1}$
an oriented connected submanifold of $\overline{M}^{n+1}$and $\Sigma
^{n-1}\subset P^{n}$ a compact hypersurface of $P^{n}.$ Let $\varphi
:M^{n}\longrightarrow \overline{M}^{n+1}$ be a compact oriented hypersurface
of boundary $\partial M$.

Let $p\in \Sigma ^{n-1}$ and $\left\{ e_{1},...,e_{n-1}\right\} $ an
orthonormal basis of $T_{p}\Sigma ^{n-1}$. We can choose a global vector
field $\nu $ such that $\left\{ e_{1},...,e_{n-1},\nu (p)\right\} $ is an
orthonormal basis of $T_{p}M^{n}$. Let $N$ be the globally vector normal to $%
M^{n}$, then $\left\{ e_{1},...,e_{n-1},\nu (p),N\right\} $ is an
orthonormal basis of $T_{p}\overline{M}^{n+1}$\bigskip . Suppose now the
existence of a closed conformal vector field $Y$ on $\overline{M}^{n+1}$;
that is to say there exists a $\phi \in C^{\infty }(\overline{M}^{n+1})$
such that 
\begin{equation}
\overline{\triangledown }_{V}Y=\phi V  \notag
\end{equation}%
for every vector fields $V$ over $\overline{M}^{n+1}.$

If $\left\{ e_{1},...,e_{n}\right\} $ is an orthonormal basis of $T_{p}M^{n}$
that diagonalizes $A,$then 
\begin{eqnarray*}
\left\langle \func{div}_{f}T_{k}^{\infty },Y\right\rangle &=&\left\langle
e^{f}\func{div}\left( e^{-f}T_{k}^{\infty }\right) ,Y\right\rangle \\
&=&e^{f}\func{div}\left( e^{-f}T_{k}^{\infty }Y\right) -\underset{i=0}{%
\overset{n}{\dsum }}\left\langle T_{k}^{\infty }\left( e_{i}\right)
,\triangledown _{e_{i}}Y\right\rangle \\
&=&e^{f}\func{div}\left( e^{-f}T_{k}^{\infty }Y\right) -\underset{i=0}{%
\overset{n}{\dsum }}\phi \left\langle T_{k}^{\infty }\left( e_{i}\right)
,e_{i}\right\rangle \\
&=&\func{div}_{f}\left( T_{k}^{\infty }Y\right) -\phi trT_{k}^{\infty }.
\end{eqnarray*}%
And in virtue of formula \ref{2'} we have 
\begin{eqnarray*}
trT_{k}^{\infty } &=&n\sigma _{k}^{\infty }(\mu _{0},\mu )-tr\left(
AT_{k-1}^{\infty }(\mu _{0},\mu )\right) \\
&=&\left( n-k\right) \sigma _{k}^{\infty }(\mu _{0},\mu )+\mu _{0}\sigma
_{k-1}^{\infty }(\mu _{0},\mu ) \\
&=&\left( n-k\right) \left( 
\begin{array}{c}
n \\ 
k%
\end{array}%
\right) H_{k,f}+\left\langle \nu ,\triangledown f\right\rangle \left( 
\begin{array}{c}
n \\ 
k-1%
\end{array}%
\right) H_{k-1,f}.
\end{eqnarray*}

So,%
\begin{equation}
\func{div}_{f}\left( T_{k}^{\infty }Y\right) =\left\langle \func{div}%
_{f}T_{k}^{\infty },Y\right\rangle +\phi \left(
c_{k}H_{k,f}+c_{k-1}\left\langle \nu ,\triangledown f\right\rangle
H_{k-1,f}\right)  \notag
\end{equation}

where $c_{k}=(n-k)\dbinom{n}{k}$ and $c_{k-1}=n\dbinom{n}{k-1}$

Integrating the two sides of this latter equality and applying the
divergence theorem, we obtain for $1\leq k\leq n-1$,%
\begin{align*}
\underset{M^{n}}{\int }\func{div}_{f}\left( T_{k}^{\infty }Y\right) dv_{f}& =%
\underset{M^{n}}{\int }e^{-f}\func{div}_{f}\left( T_{k}^{\infty }Y\right) dv
\\
& =\underset{M^{n}}{\int }\func{div}\left( e^{-f}T_{k}^{\infty }Y\right) dv
\\
& =\underset{\partial M}{\int }e^{-f}\left\langle T_{k}^{\infty }Y,\nu
\right\rangle ds \\
& =\underset{\partial M}{\int }\left\langle T_{k}^{\infty }Y,\nu
\right\rangle ds_{f}.
\end{align*}

Hence,%
\begin{equation*}
\underset{\partial M}{\int }\left\langle T_{k}^{\infty }\nu ,Y\right\rangle
ds_{f}=\underset{M^{n}}{\int }\left\langle \func{div}_{f}T_{k}^{\infty
},Y\right\rangle dv_{f}+c_{k}\underset{M^{n}}{\int }\phi
H_{k,f}dv_{f}+c_{k-1}\underset{M^{n}}{\int }\phi \left\langle \nu
,\triangledown f\right\rangle H_{k-1,f}dv_{f}\text{.}
\end{equation*}%
Consequently, we have the following proposition

\begin{proposition}
\label{prop2} Let $\varphi :M^{n}\longrightarrow \overline{M}^{n+1}$ an
immersed compact oriented hypersurface of boundary $\partial M$. Denoting by 
$N$ \ the global vector fields normal to $M^{n}$, and $\nu $ the outward
pointing conormal vector field to $M^{n}$ along $\partial M$. Then for $%
1\leq k\leq n-1$ and for every closed conformal vector field $Y$ on $%
\overline{M}^{n+1}$, we have :%
\begin{equation*}
\underset{\partial M^{n}}{\int }\left\langle T_{k}^{\infty }\nu
,Y\right\rangle ds_{f}=\underset{M^{n}}{\int }\left\langle \func{div}%
_{f}T_{k}^{\infty },Y\right\rangle dv_{f}+c_{k}\underset{M^{n}}{\int }\phi
H_{k,f}dv_{f}+c_{k-1}\underset{M^{n}}{\int }\phi \left\langle \nu
,\triangledown f\right\rangle H_{k-1,f}dv_{f}.
\end{equation*}%
\ \ \ \ \ \ \ \ \ \ \ \ \ \ \ \ \ \ \ \ \ \ \ \ \ \ \ \ \ \ \ \ \ \ \ \ \ \
\ \ \ \ \ \ \ \ \ \ \ \ \ \ \ \ \ \ \ \ \ \ \ \ \ \ \ \ \ \ \ \ \ \ \ \ \ \
\ \ \ \ \ \ \ \ \ \ \ \ \ \ \ \ \ \ \ \ \ \ \ \ \ \ \ \ \ \ \ \ \ \ \ \ \ \
\ \ \ \ \ \ \ \ \ \ \ \ \ \ \ \ \ \ \ \ \ \ \ \ \ \ \ \ \ \ \ \ \ \ \ \ \ \
\ \ \ \ \ \ \ \ \ \ \ \ \ \ \ \ \ \ \ \ \ \ \ \ \ \ \ \ \ \ \ \ \ 
\end{proposition}

If $\overline{M}^{n+1}$ has constant sectional curvature, we obtain by
Corollary \ref{cor},

\begin{proposition}
Under the hypothesis of the proposition \ref{prop2}, if $\overline{M}^{n+1}$
has constant sectional curvature, then 
\begin{equation}
\underset{\partial M}{\int }\left\langle T_{k}^{\infty }\nu ,Y\right\rangle
ds_{f}=\underset{M^{n}}{\int }\left\langle \nabla f,T_{k}^{\infty
}Y\right\rangle dv_{f}+c_{k}\int \phi H_{k,f}dv_{f}+c_{k-1}\int \left( \phi
\mu _{1}+\frac{1}{n}\left\langle \nabla \mu _{1},Y\right\rangle \right)
H_{k-1,f}dv_{f}
\end{equation}%
with $\mu _{1}=\left\langle \nabla f,\nu \right\rangle $.
\end{proposition}

\begin{corollary}
Under the hypothesis of the proposition \ref{prop2}, if $\overline{M}^{n+1}$
has constant sectional curvature and $f$ is constant then 
\begin{equation}
\underset{\partial M}{\int }\left\langle T_{k}\nu ,Y\right\rangle
ds_{f}=c_{k}\underset{M^{n}}{\int }\phi H_{k}dv_{f}.  \label{3}
\end{equation}%
where $T_{k}$ is the classical Newton transformation.
\end{corollary}

If $f$ is constant, $H_{k}$ is a non zero constant and $Y$ is an homothetic
vector field, we can assume that $\phi =1,$ and get 
\begin{equation*}
c_{k}H_{k}vol(M^{n})=\underset{\partial M}{\int }\left\langle T_{k}^{\infty
}\nu ,Y\right\rangle ds
\end{equation*}%
or equivalently 
\begin{equation*}
vol\left( M^{n}\right) =\frac{1}{c_{k}H_{k}}\underset{\partial M}{\int }%
\left\langle T_{k}^{\infty }\nu ,Y\right\rangle ds.
\end{equation*}

\begin{proposition}
Let $\varphi :M^{n}\longrightarrow \overline{M}^{n+1}$ be an immersed
compact oriented hypersurface of boundary $\partial M$. Denoting by $N$ the
global vector fields normal to $M^{n}$, and $\nu $ the outward pointing unit
conormal vector field to $M^{n}$ along $\partial M$. Suppose that $\overline{%
M}^{n+1}$ is of constant sectional curvature, $f$ is constant and the $%
k^{th} $ mean curvature $H_{k}$ is non zero constant. Then for $1\leq k\leq
n-1$ and for every homothetic vector field $Y\in \overline{M}^{n+1})$, we
have 
\begin{equation}
vol\left( M^{n}\right) =\frac{1}{c_{k}H_{k}}\underset{\partial M}{\int }%
\left\langle T_{k}^{\infty }\nu ,Y\right\rangle ds.  \label{4}
\end{equation}
\end{proposition}

\begin{remark}
An estimate of the integrant in formula \ref{4} leads to an estimate of the
volume of $M^{n}$ by the area of its boundary$\ \partial M.$
\end{remark}

\end{document}